\documentclass[12pt]{article}
\usepackage{amsfonts}
\usepackage{amsmath}
\usepackage{amsthm}
\usepackage{xypic}

\newcommand{\seq}{\textnormal{seq}}

\numberwithin{equation}{section}

\newtheorem{theorem}{\bf Theorem}[section]
\newtheorem{corollary}[theorem]{\bf Corollary}
\newtheorem{lemma}[theorem]{\bf Lemma}
\newtheorem{proposition}[theorem]{\bf Proposition}

\title{Further Properties of Random Threshold Graphs} 
\author{Christopher Ross} 

\begin{document}
\maketitle

\begin{abstract}

\hspace{.2in} 

In 2009, two different groups independently explored the behavior of random threshold graphs.  Here, we extend their techniques to find the distribution of other properties, including matching number, degeneracy, and length of the longest cycle.

\end{abstract}

\everymath{\displaystyle}

\section{Introduction}

An undirected graph $G$ is a threshold graph if there exists some real-valued function $w$ that assigns weights to the vertex set $V(G)$ such that two vertices $u, v$ are adjacent if and only if $w(u)+w(v)$ exceeds some threshold $t$.  These graphs, first defined in 1973 by Chv\'atal and Hammer \cite{Sppatg}, also have several other equivalent characterizations, which led to their occasional ``rediscovery'' through the following two decades.

In 2009, it was independently shown by Reilly and Scheinerman \cite{RTG}, as well as by Diaconis, Holmes, and Janson \cite{TGLaRTG}, that the method of generating random threshold graphs by choosing $n$ vertex weights uniformly on $[0,1]$ (with $t = 1$) was in fact uniform on the set of all $n$-vertex threshold graphs.  The two teams then used this equivalence to find properties ranging from the distribution of the number of isolated vertices to the likelihood of Hamiltonicity.

Here, we take their results and extend them, using the encodable nature of threshold graphs to determine the distributions and likelihoods of other graph invariants.

\section{Basics}

One of the many equivalent characterizations of threshold graphs is that they can be constructed from a single vertex by repeatedly adding an isolated vertex or a dominating vertex \cite{Aoiiip, TGaRT}.  So a threshold graph on $n$ vertices is completely determined by this record of $n-1$ additions; if we mark a $0$ for the addition of an isolated vertex and a $1$ for the addition of a dominating vertex, we get a binary sequence which is known as the \emph{creation sequence}.  (This definition, drawn from \cite{RTG}, is equivalent to the \emph{binary code} defined in \cite{TGLaRTG}.  It differs slightly from the creation sequence definition of \cite{Dtnwgsadp}, which is closer to the \emph{extended binary code} of \cite{TGLaRTG}, as both allocate an extra digit for the original single vertex.)

Given a threshold graph $G$ with $n$ vertices, we let $\seq(G)$ denote the $(n-1)$-digit creation sequence of $G$.  Conversely, given a binary string $s$ of length $n$, $\gamma(s)$ is the unlabeled threshold graph $G$ such that $\seq(G) = s$.  From this, we see that the number of $n$-vertex threshold graphs is exactly $2^{n-1}$.  

These properties suggest two natural methods for random generation of a threshold graph with $n$ vertices.  The first is to choose the $n$ weights independently and uniformly from $[0,1]$ with threshold $t=1$, and let $G$ denote the unlabeled threshold graph induced by the weights; in this model, the probability of any particular edge ${u,v}$ being in the graph is exactly $1/2$.  Alternatively, we can select $G$ uniformly at random from the set of all $2^{n-1}$ possible graphs of given size.  

A critical result in the study of random threshold graphs was the 2009 proof that these two methods have the same distribution.  That is, via independent arguments, collaborations between Diaconis, Holmes, and Janson \cite{TGLaRTG} and between Reilly and Scheinerman \cite{RTG} showed that given an unlabeled threshold graph $G$ with $n$ vertices, the probability of $G$ arising via the random selection of vertex weights drawn independently and uniformly from $[0,1]$, with threshold $t=1$, is equal to $2^{1-n}$.  

As a consequence, when examining random threshold graphs, we can discard the continuous random variables involving vertex weights and restrict ourselves to the discrete creation sequences.  So to compute the probability $P$ of random threshold graph $G$ having a certain invariant, we first find the properties of the creation sequence necessary and sufficient to evoke such behavior in $G$.  Then $P$ will be proportional to the number of $(n-1)$-long binary sequences with such properties, a number that is usually easier to count.

With the exception of that single initial vertex, also known as the \emph{base vertex}, every element of $V(G)$ can be classified according to its digit in $\seq(G)$.  We call the others zero-vertices or one-vertices, depending upon whether the corresponding digit is zero or one, respectively.  Furthermore, we use the creation sequence to refer to specific vertices, saying that a vertex has \emph{index} $i$ if it corresponds to the $i$-th digit in the creation sequence, reading from left to right.  (The base vertex has index $0$.)  Note that this enumeration is a by-product of the structure of the graph, as opposed to an independent labelling.

As a consequence of this construction, the relationship between any two vertices can be completely determined by their corresponding digits and their relative indices.  Since one-vertices dominate all existing vertices at the time of their addition, and zero-vertices are isolated from all existing vertices, two vertices are adjacent if and only if the vertex of higher index has a corresponding digit of $1$.  So no two zero-vertices are adjacent, but all one-vertices are adjacent to each other, as well as to the base vertex.

\begin{proposition} For any two threshold graphs $G$ and $H$, $G$ is an induced subgraph of $H$ if and only if $\seq(G)$ is a subsequence of $\seq(H)$.
\end{proposition}
\begin{proof}

Letting $\seq(G) = s_1 s_2 \dots s_m$ and $\seq(H) = t_1 t_2 \dots t_n$, suppose that there exist $j_1 < j_2 < \dots < j_m$ such that $\seq(G) = t_{j_1} t_{j_2} \dots t_{j_m}$.  For any $u, v \in V(G)$ of index $i_u$ and $i_v$, let $u', v' \in V(H)$ be those vertices of index $j_{i_u}$ and $j_{i_v}$.  Then as the corresponding digits of $u'$ and $v'$ are equal to, and in the same order as, those of $u$ and $v$, the former are adjacent in $H$ if and only if the latter are adjacent in $G$.  Thus, $G$ must be an induced subgraph of $H$.

For the converse, it suffices to prove the claim for when $G$ is induced by the removal of a single vertex of $H$; repeated application produces all other cases.  Let $G$ be the induced subgraph of $H$ formed by the removal of vertex $v$, and consider the construction of $\seq(G)$.  The removal of $v$ does not change the classification of any vertices: if vertex $u$ is isolated in $H$, it remains isolated in $G$.  Similarly, if $u$ dominated all lower-index vertices in $H$, it dominates all of those vertices that remain in $G$.  So when forming $\seq(G)$, the vertices can be taken in the same order as when forming $\seq(H)$, without changing their corresponding digits. 
\end{proof}

Given a sequence $s = s_1 s_2 \dots s_n$, we let the $k$-th tail of $s$ be the subsequence consisting of the last $k$ digits of $s$: $s_{n-k+1} s_{n-k+2} \dots s_n$.  In this vein, we define $z_k(s)$ and $u_k(s)$ to be the numbers of zeros and ones in the $k$-th tail, respectively.

We define a function $h$ on the set of all finite binary sequences by, for any such sequence $s$,
$$ h(s) = \max_{ 0 \leq k \leq |s| } \{ z_k(s) - u_k(s) \} $$
So $h(s)$ is the maximum count, across all tails of $s$, by which the number of zeros exceeds the number of ones.  Note that $h$ is always non-negative, as the case $k=0$ corresponds to the empty tail, in which there are no digits of either type.

To make full use of $h$, we must first find its distribution.

\begin{proposition} For a random threshold graph $G$ on $n$ vertices,
$$ P(h(\seq(G)) = k ) = \left( \frac{1}{2} \right)^{n-1} \binom{ n-1 }{ \left\lfloor \frac{n+k}{2} \right\rfloor }  $$
\end{proposition}
\begin{proof}
Using the uniformity of the distribution, we see that the probability of $h(\seq(G)) = k$ is proportional to the number of $(n-1)$-long binary sequences that have a tail with $k$ more zeros than ones, but where no tail has $k+1$ more zeros than ones.  

To count the number of such sequences, we read the creation sequences from right to left, and interpret the digits as moves within an integer lattice.  Starting at the origin, we move a single unit upwards whenever we encounter a zero, and rightwards whenever we encounter a one.  In this framework, a tail with $m+k$ zeros and $m$ ones produces a ``staircase walk'' from the origin to the point $(m,m+k)$.  

So the number of $(n-1)$-long sequences $s$ such that $h(s) = k$ equals the number of staircase walks that touch, but do not cross, the line $y = x + k$.  Thus,
$$ P(h(\seq(G)) = k) = \left( \frac{1}{2} \right)^{n-1} \binom{ n-1 }{ n-1+k - \left\lfloor \frac{n+k-1}{2} \right\rfloor } $$
\end{proof}

\section{Planarity}

By Kuratowski's Theorem, $G$ is planar if and only if it does not contain a subgraph that is a subdivision of $K_5$ or $K_{3,3}$.  As such, we will use the following result which shows exactly when $G$ has $K_5$ as a subgraph:

\begin{lemma}[Reilly, Scheinerman]\label{ThreshClique} For a threshold graph $G$, the size of the maximum clique is one more than the number of one-vertices.
\end{lemma}

\begin{proposition}\label{ThreshPlanar} A threshold graph $G$ with $s = \seq(G)$ is planar if and only if $s$ contains no subsequence of the form $1111$ or $00111$.  
\end{proposition}
\begin{proof}
We shall show that the existence of Kuratowski's offending subgraphs in $G$ is equivalent to the existence of the above subsequences in $s$.

First, suppose that $s$ contains a subsequence of the form $1111$.  Then by Lemma \ref{ThreshClique}, $G$ contains $K_5$ as an induced subgraph.  Alternatively, if $s$ contains some $00111$, then there must exist three one-vertices that are each adjacent to three other vertices: the vertices corresponding to the two zeros and the base vertex.  Thus, $G$ contains $K_{3,3}$ as a subgraph.  In either of these cases, we see that $G$ is non-planar.

Inversely, suppose that $s$ contains no such subsequences, which leads to two subcases: either there are at most two ones in $s$, or there exist exactly three ones, one of which has index at most two.  In the former event, there are at most two vertices of degree exceeding two.  Since subdividing does not increase the degree of existing vertices, no subdivision of any subgraph can have three vertices of degree three or more, eliminating $K_5$ and $K_{3,3}$ as possibilities.  

Similarly, if $G$ contains exactly three one-vertices, one of which has index at most two, then said one-vertex has degree at most four.  There are at most five vertices of degree three or more: the three one-vertices, and any other vertices of lesser index, of which there are at most two.  Thus, no subgraph subdivides into $K_{3,3}$.  And as the base vertex and the zero-vertices can have degree at most three, a subdivision into $K_5$ is also impossible.  Thus, $G$ must be planar.
\end{proof}

\begin{theorem} For a random threshold graph $G$ with $n$ vertices, 
$$ P( G \textnormal{ is planar}) = \left\{ \begin{array}{cc} 1 & n \leq 4 \\ 
\frac{ 3n^2 - 13n + 20 }{ 2^n } & n \geq 4 \end{array} \right. $$
\end{theorem}
\begin{proof}
Letting $s = \seq(G) = s_1 s_2 \dots s_{n-1}$, we see that by Proposition \ref{ThreshPlanar}, $G$ is planar if and only if $s$ contains at most two ones, or $s$ has three ones, one of which must be $s_1$ or $s_2$.  By counting the number of such sequences, the probability of the former event is
$$ \left( \frac{1}{2} \right)^{n-1} \left[ \binom{n-1}{0} + \binom{n-1}{1} + \binom{n-1}{2} \right] $$

For the latter event, we further subdivide into the disjoint events $\{s_1 = 1\}$ and $\{s_1 = 0, s_2 = 1\}$, which have a combined probability of
$$ \left( \frac{1}{2} \right)^{n-1} \left[ \binom{n-2}{2} + \binom{n-3}{2} \right] $$
\end{proof}

\section{Matching Number}

Reilly and Scheinerman found the probability of a random threshold graph having a perfect matching.  Here, we explore the distribution of the matching number $\nu(G)$, the number of edges in a maximum matching.

\begin{lemma}[Reilly, Scheinerman]\label{RSPerfMatch} A threshold graph $G$ with an even number of vertices contains a perfect matching if and only if $h(\seq(G)) = 0$.
\end{lemma}

\begin{corollary} A threshold graph $G$ with an odd number of vertices has a near-perfect matching if $h(\seq(G)) = 0$.
\end{corollary}
\begin{proof}
Letting $\seq(G) = s_1 s_2 \dots s_k$, define subsequence $s'$ by $s' = s_2 s_3 \dots s_k$.  Then $h(s') = 0$, as every tail of $s'$ is also a tail of $\seq(G)$.  Since $\gamma(s')$ is a threshold graph with one fewer vertex than $G$, it has a perfect matching, which is also a matching in $G$.
\end{proof}

\begin{proposition} For a threshold graph $G$ with $n$ vertices, 
$$ \nu(G) = \left\lfloor \frac{ n - h( \seq(G) ) }{ 2 } \right\rfloor $$
\end{proposition}
\begin{proof}

The case $h(\seq(G)) = 0$ having already been handled, we assume $h(\seq(G)) \geq 1$.  Let $s = \seq(G)$.

As $2 \nu(G)$ is the maximum number of vertices in a matching, $n - 2 \nu(G)$ is the minimum across all matchings of the number of unpaired vertices.  Let $m$ denote a maximizing index for $h(s)$, so that $h(s) = z_m(s) - u_m(s)$.  Then there are $h(s)$ more zeros than ones amongst the final $m$ digits.  As zero-vertices are adjacent only to one-vertices of higher index, those vertices corresponding to zeros in the tail can only be adjacent to one-vertices in the same tail, so there are at least $h(s)$ vertices that cannot participate in any matching.  Thus, $n - 2 \nu(G) \geq h(s)$.

Next, let us define a new binary sequence $s'$ by removing the $h(s)$ right-most zeros from $s$.  Then $h(s') = 0$, so by Lemma \ref{RSPerfMatch} the threshold graph $\gamma(s')$, which has $n - h(s)$ vertices, has a matching of size $\left\lfloor (n - h(s)) / 2 \right\rfloor$.  And as $\gamma(s')$ is an induced subgraph of $G$, $\nu(G) \geq \nu(\gamma(s))$.
\end{proof}

Having determined the properties of the creation sequence responsible for a matching number of given size, we can compute its likelihood.

\begin{theorem} For a random threshold graph $G$ with $n$ vertices,
$$ P( \nu(G) = k ) = \left\{ \begin{array}{cc} \left( \frac{1}{2} \right)^{n-1} \binom{n}{k} & k < \frac{n}{2} \\ \left( \frac{1}{2} \right)^{n-1} \binom{n-1}{\left\lfloor \frac{n-1}{2} \right\rfloor } & k = \frac{n}{2} \end{array} \right. $$
\end{theorem}
\begin{proof}
Letting $s$ denote $\seq(G)$, we note that since $h$ can only assume integer values,
\begin{eqnarray*} P( \nu(G) = k ) & = & P(n - 2k - 1 \leq h(s) \leq n - 2k ) \\
& = & P( h(s) = n-2k-1) + P(h(s) = n-2k) \end{eqnarray*}
As $h(s)$ must be non-negative, we see that for $0\leq k < n/2$,
\begin{eqnarray*} P( \nu(G) = k ) & = & \left( \frac{1}{2} \right)^{n-1} \left( \binom{n-1}{k} + \binom{ n-1 }{ k-1 } \right) \\
& = & \left( \frac{1}{2} \right)^{n-1} \binom{n}{k} \end{eqnarray*}

As for $k = n/2$, we see that
$$ P \left( \nu(G) = \frac{n}{2} \right) =  P \left( h(s) = 0 \right) = \left( \frac{1}{2} \right)^{n-1} \binom{ n-1 }{ \left\lfloor \frac{ n}{2}  \right\rfloor } $$ 
\end{proof}

\section{Longest Cycle Length}

For a graph $G$, let $\psi(G)$ denote the length of the longest cycle in $G$.  For a graph $G$ on $n$ vertices, Reilly and Scheinerman found the probability that $\psi(G) = n$, corresponding to the event in which $G$ is Hamiltonian, through the following result:

\begin{lemma}[Reilly, Scheinerman] Let $G$ be a threshold graph with $n \geq 3$ vertices.  Then $G$ is Hamiltonian if and only if $u_k(\seq(G)) > z_k(\seq(G))$ for all $1 \leq k \leq n-1$.
\end{lemma}

\begin{corollary}\label{HamCor} Let $G$ be a threshold graph with $n \geq 3$ vertices, and $\seq(G) = s_1 s_2 \dots s_{n-1}$.  Then $G$ is Hamiltonian if and only if $s_{n-1} = 1$ and $h(s_1 s_2 \dots s_{n-2}) = 0$.
\end{corollary}

Here, we generalize to find the full distribution of $\psi(G)$.  We define, for a binary sequence $s = s_1 s_2 \dots s_k$, the function $r(s)$ by
$$ r(s) = \max \left( \{0\} \cup \{ i : s_i = 1 \} \right) $$
That is, $r(s)$ returns the index of the right-most one in $s$ in the event that such exists, and zero otherwise.  

\begin{proposition}\label{PsiSize} For a threshold graph $G$ with $s = \seq(G)$, if $s$ contains at least two ones, then
$$ \psi(G) = r(s) + 1 - h(s'),$$
where $s'$ is the subsequence of $s$ defined by $s' = s_1 s_2 \dots s_{r(s)-1}$.
\end{proposition}
\begin{proof}

First, note that this formulation, like that for $\nu(G)$, can be expressed in terms of excluded vertices $n - \psi(G)$.  To find the minimum of the number of vertices skipped by each cycle, we begin by excluding the isolated vertices, of which there are exactly $(n-1-r(s))$.  

As for the non-trivial connected component, which corresponds to the sequence $s'1$, let $m$ be a maximizing index for $h(s')$, so that $h(s') = z_m(s') - u_m(s')$.  Then no cycle can contain more than $2 u_m(s')$ of the vertices corresponding to the last $m$ digits.  For if some cycle $Y$ were to contain $u_m(s') +1$ of the zero-vertices, then the threshold subgraph $H_Y \subseteq G$ induced by the vertices of $Y$ would have a Hamiltonian cycle but a creation sequence where some tail contained as many zeros as ones, contradicting Corollary \ref{HamCor}.  So at least $h(s')$ of the zero-vertices in the connected component must also be excluded, and thus
$$ n - \psi(G) \geq (n-1-r(s)) + h(s') $$

Let us define another binary sequence $s''$ by removing the right-most $h(s')$ zeros from $s'$; then $h(s'') = 0$ and $|s''| = r(s)-1-h(s')$.  Therefore $\gamma(s''1)$, an induced subgraph of $G$, is a threshold graph on $r(s) + 1 -h(s')$ vertices that contains a Hamiltonian cycle.  Thus, $\psi(G) \geq r(s) + 1 - h(s')$.
\end{proof}

\begin{theorem} For a random threshold graph $G$ on $n$ vertices,
$$ P( \psi(G) = k ) = \left( \frac{1}{2} \right)^{n-1} \left[ \binom{n-1}{ \left\lfloor \frac{k}{2} \right\rfloor } - \binom{k-2}{ \left\lfloor \frac{k}{2} \right\rfloor } \right] $$
\end{theorem}
\begin{proof}

By Proposition \ref{PsiSize}, in order for $\psi(G)$ to equal $k$, we require that
$$ r(s) - k + 1 = h(s_1 s_2 \dots s_{r(s)-1}),$$
where $s = \seq(G)$.  So for $3 \leq k \leq n$, 
\begin{eqnarray*}
P( \psi(G) = k) & = & \sum_{ j = 1 }^{ n-1 } P( r(s) = j, \psi(G) = k ) \\
& = & \sum_{ j = 1 }^{ n-1 } P( r(s) = j, r(s) - k + 1 = h(s_1 s_2 \dots s_{r(s)-1}) ) \\
& = & \sum_{ j = k-1 }^{ n-1 } P( r(s) = j, j - k + 1 = h(s_1 s_2 \dots s_{j-1}) ) 
\end{eqnarray*}

Note that the two intersecting events are independent: the first depends only on the location of the right-most one in the sequence, whereas the second concerns all of the preceeding digits, and there is no overlap.  And since each individual digit is chosen independently, the distribution of $s_1 s_2 \dots s_{r(s)-1}$, conditioned on $r(s)$, is uniform over all binary sequences of length $r(s)-1$:

\begin{eqnarray*}
P( \psi(G) = k) & = & \sum_{ j = k-1 }^{ n-1 } P( r(s) = j) P( h(s_1 s_2 \dots s_{j-1} ) =  j - k + 1  ) \\
& = & \sum_{ j = k-1 }^{ n-1 } \left( \frac{1}{2} \right)^{n-j} \left( \frac{1}{2} \right)^{j-1} \binom{ j - 1 }{ 2j -k - \left\lfloor \frac{2j-k}{2} \right\rfloor } \\
& = & \left( \frac{1}{2} \right)^{n-1} \sum_{ j = k-1 }^{ n-1 }  \binom{ j - 1 }{ j - \left\lfloor \frac{k}{2} \right\rfloor } \\
& = & \left( \frac{1}{2} \right)^{n-1} \left[ \binom{n-1}{ \left\lfloor \frac{k}{2} \right\rfloor } - \binom{k-2}{ \left\lfloor \frac{k}{2} \right\rfloor } \right] 
\end{eqnarray*}
\end{proof}

\section{$k$-Core}

The $k$-core of a graph $G$ is the maximum induced subgraph $H \subseteq G$ such that all vertices of $H$ have degree at least $k$, formed by iteratively deleting all vertices with degree less than $k$.  The \emph{degeneracy} of $G$ is the largest $k$ such that the $k$-core of $G$ is non-empty.  An equivalent formulation for the degeneracy is the maximum, over all induced subgraphs $H \subseteq G$, of the minimum degree of a vertex in $H$.  That is,
$$ degen(G) = \max_{H \subseteq G} \min_{v \in V(H)} \deg(v) $$

These two concepts allow us to examine the density of our randomly generated graphs:
\begin{proposition} For a threshold graph $G$, $degen(G) \geq d$ if and only if $K_{d+1} \subseteq G$.
\end{proposition}
\begin{proof}

First, assume that $G$ contains a subgraph isomorphic to $K_{d+1}$.  Then letting $H$ denote said subgraph, we see that every vertex has degree exactly $d$, and thus the degeneracy is at least $d$, by the second definition given above.

Next, assume that $G$ has degeneracy greater than or equal to $d$.  Then $G$ has a non-empty $d$-core, and therefore there exists a subset $V' \subseteq V(G)$ in which every vertex of $V'$ is adjacent to at least $d$ other members of $V'$.  Then $|V'| \geq d+1$, and since zero-vertices are adjacent only to one-vertices, $V'$ must contain at least $d$ of $G$'s one-vertices.  So $G$ contains at least $d$ one-vertices, and thus a clique of size $d+1$, as all one-vertices are adjacent to each other, as well as the base vertex.
\end{proof}

\begin{corollary} For a threshold graph $G$, $degen(G) = d$ if and only if $\seq(G)$ contains exactly $d$ ones.
\end{corollary}

\begin{corollary}\label{degenDist} For a random threshold graph $G$ with $n$ vertices, 
$$ P( degen(G) = d ) = \left( \frac{1}{2} \right)^{n-1} \binom{n-1}{d} $$
\end{corollary}

\begin{proposition} For a threshold graph $G$ such that $s = \seq(G)$ contains at least $k$ ones, a vertex $v \in V(G)$ is in the $k$-core of $G$ if and only if $\deg(v) \geq k$.
\end{proposition}
\begin{proof}

Since $G$ contains at least $k$ one-vertices, $G$ contains a non-empty $k$-core.  If a vertex lies in the $k$-core of $G$, then by definition it has a degree of at least $k$ in that induced subgraph, and therefore a degree of at least $k$ in $G$.

Next, consider some vertex $v$ such that $\deg(v) \geq k$.  If $v$ is a one-vertex, then $v$ is a vertex of some induced $K_{k+1}$, and thus part of the $k$-core.  On the other hand, if $v$ is a base vertex or zero-vertex, then $v$ is adjacent to $k$ one-vertices in the $k$-core, and thus part of the $k$-core as well.
\end{proof}

Because of this, if a non-empty $k$-core exists, then only one round of pruning occurs.  Furthermore, the pruned vertices are all zero-vertices of degree less than $k$.

\begin{theorem} For a random threshold graph $G$ on $n$ vertices,
$$ P( |k-core(G)| = j ) = \left\{ \begin{array}{cc} 
\sum_{i=0}^{k-1} \left( \frac{1}{2} \right)^{n-1} \binom{n-1}{i} & j = 0 \\
\left( \frac{1}{2} \right)^{n+k-j} \binom{n+k-j-1}{k-1} & j \geq k+1 \end{array} \right. $$
\end{theorem}
\begin{proof}

For the $k$-core of $G$ to be empty, the degeneracy of $G$ can be at most $k-1$.  So by Corollary \ref{degenDist} the probability of having $|k\textnormal{-core}(G)| = 0$ is
$$ \sum_{i=0}^{k-1} P( degen(G) = i ) = \sum_{i=0}^{k-1} \left( \frac{1}{2} \right)^{n-1} \binom{n-1}{i} $$

In the non-empty cases, the event of the $k$-core having exactly $j$ vertices is the same as exactly $n-j$ vertices being removed by the pruning process.  These discards, being zero-vertices of low degree, can have at most $k-1$ one-vertices of higher index.  So $\seq(G)$ has $n-j$ zeros lying to the right of the $k$-th one from the right.

To summarize, the $k$-core of $G$ has exactly $j > 0$ vertices if and only if the right-most $(n+k-j-1)$ digits of $\seq(G)$ contain exactly $k-1$ ones and $n-j$ zeros, and the $(n+k-j)$-th digit from the right is a one.  As there are no restrictions on the first $j-k-1$ digits, there are $2^{j-k-1} \binom{n+k-j-1}{k-1}$ such sequences.  Therefore,
$$ P(|k\textnormal{-core}(G)|=j) = \frac{2^{j-k-1}}{2^{n-1}} \binom{n+k-j-1}{k-1} $$
\end{proof}

\clearpage 


\end{document}